\documentclass[leqno]{amsart}
\usepackage[scale=0.70, centering, headheight=14pt]{geometry}
\usepackage{amsfonts,amssymb,amsmath,amsgen,amsthm}
\usepackage{fancyhdr}
\usepackage{color}
\usepackage[svgnames]{xcolor}
\xdefinecolor{color}{named}{Crimson}
\usepackage{tikz}
\usepackage{array}

\theoremstyle{plain}
\newtheorem{theorem}{Theorem}[section]
\newtheorem{definition}[theorem]{Definition}

\newtheorem{proposition}[theorem]{Proposition}
\theoremstyle{remark}
\newtheorem{remark}[theorem]{Remark}
\newtheorem*{notation}{Notation}

\def\eps{\varepsilon}
\def\C{{\mathbf C}}
\def\R{{\mathbf R}}
\def\N{{\mathbf N}}
\def\Sch{{\mathcal S}}
\def\FT{\mathcal F}
\def\d{{\partial}}

\newcommand{\myitem}{\item[$\star$]}

\numberwithin{equation}{section}

\title{Small data scattering for energy-subcritical and critical Nonlinear Klein Gordon equations on product spaces}
\author[]{Lysianne HARI \and Nicola VISCIGLIA
}
\address[]{ \newline (L. HARI \and N. VISCIGLIA) 
\newline Dipartimento di Matematica, Universit\`a di Pisa, \newline Largo Bruno Pontecorvo 5, \newline 56127 Pisa, Italia.
  \newline                }
\email{lhary@mail.dm.unipi.it,}
\email{viscigli@dm.unipi.it}

\begin{document}

\begin{abstract}
We study small data scattering of solutions to Nonlinear Klein-Gordon equations with suitable pure power nonlinearities, 
posed on $\R^d\times \mathcal{M}^k$ with $k\leq2$ and $d\geq1$ and $\mathcal{M}^k$ a compact Riemannian manifold.
As a special case  we cover the $H^1-$critical NLKG on $\R \times M^2$.
\\

\\

\end{abstract}
\maketitle
\let\thefootnote\relax\footnotetext{This work was partially supported by ERC Grant Blowdisol and INDAM-COFUND Project: NLDISC.}

\section{Introduction}
\label{intro}
Let $d\geq 1$ and $k=1,2$, such that $3\leq d+k\leq 6$. We consider the Cauchy problem for the Nonlinear Klein-Gordon equation (NLKG) posed on 
$R^d\times \mathcal{M}^k$
\begin{equation}
\label{NLKG}
       \d^2_t u - \Delta_{x,y} u + u = \pm |u|^{p-1} u, \quad u_{|_{t=0}}(x,y)= f(x,y), \quad \d_t u_{|_{t=0}}(x,y)= g(x,y),
\end{equation}
with $(t,x,y)\in \R_t \times \R^d_x \times \mathcal{M}^k_y $, where
\begin{itemize}
 \myitem $d\geq 1$,
 \myitem $\mathcal{M}^k_y$ is a compact Riemannian manifold of dimension $k$,
 \myitem $\Delta_{x,y} = \Delta_x + \Delta_y$, where $\Delta_x = \sum_{j=1}^{d} \d^2_{x_j}$ 
 is the Laplace operator associated to the flat metric on $\R^d$ and $\Delta_y$ is the Laplace-Beltrami operator on 
 $\mathcal{M}^k_y$. We write $dy$ the volume element of $\mathcal{M}^k_y$, and 
 $$Vol(\mathcal{M}^k) = \int_{\mathcal{M}^k} \; dy $$ the (finite) volume of $\mathcal{M}^k$,
  \myitem $p$ is such that $p_0\leq p \leq p_c$ where $p_c$ is the $H^1-$critical exponent for the whole dimension $d+k$: 
 \begin{equation}
 \label{pc}
p_c = 1 + \dfrac{4}{d+k-2},  
 \end{equation}
and where 
\begin{equation}
\label{p0}
p_0 = \max \left(2,1+\dfrac{4}{d}\right) 
\end{equation}
is larger than the  $L^2-$critical exponent on $\R^d$.
\end{itemize}
\vspace{2mm}
The aim of this paper is to investigate the persistence of scattering properties for the nonlinear Klein-Gordon equation (NLKG), 
considered on product spaces of total dimension larger than three.
The study of dispersive PDE's posed on product spaces was first initiated for 
the Schr\"odinger equation (NLS) (see \cite{HTT14, TTNLS} for problems involving global well-posedness, \cite{HP14,HPTV,TarulliRdM2,TVI,TVII} 
for long time asymptotics and \cite{TTV} for studies about ground states). 
On one hand, choosing carefully the parameters,
one has global exis\-tence and scattering results on Euclidean spaces $\R^d$, in particular for small data. 
On the other hand, considering the equation on compact Riemannian manifolds, the previous assertion fails. 
A question then arises when the equation is posed on a product of Euclidean space and compact Riemannian manifold: is the dispersive nature 
of the Euclidean part sufficient to prevail and rule the behaviour of the whole solution at infinity ? 
\\

The same question comes up for other dispersive PDEs in such settings; we focus on the nonlinear Klein-Gordon equation in this paper.
A particular case is when $\mathcal{M}^k$ is $\mathbb{T}^k$ the flat torus of dimension $k$: this kind of semiperiodic settings is 
involved in wave guide theories, especially in the case $d+k=3$. Besides, there are numerous references about the cubic (NLKG) or (NLW) 
posed on specific manifolds in general relativity.
However, we will see that the topological structure of the manifold is not relevant in our study and that the small data theory does not need 
any specific property on $\mathcal{M}^k$.
\\
The analysis of existence and scattering properties on $\R^d$ can be found in several references. 
It is well-known that scattering for small $H^1$ data holds for $$1+\dfrac{4}{d}\leq p \leq 1+ \dfrac{4}{d-2}$$
since existence of wave operators and asymptotic completeness is proved in those cases.
\\
We do not make any exhaustive list of all previous works dealing with similar problems but
for the particular case of small data scattering results, we refer the reader to 
\cite{Brenner85, GV_NLKG_I, GV_NLKG_II, Pecher84, Pecher85, Straussbook, TsuHay84}. All dimensions and all $p$ lying between 
the $L^2-$critical and the $H^1-$critical exponents are not necessarily handled in these references.
\\
Among several papers dealing with small data scattering in $\R^d$ for (NLW) (case $m^2=0$), we give 
\cite{LS95, NOpisa, MM85, MM87, MMbook, Pecher88, Strauss81, Strauss81sequel} (some of the references also contain the case $m^2>0$) 
and references therein.
\\
More recent results can be found in \cite{Nakamura03, Wang98, Wang99}, and similar problems in different frameworks are discussed in 
\cite{NOpisa}, where the nonlinearity is a sum of nonlinear terms with exponents from $L^2-$critical 
to $H^s-$critical nonlinearities with $s<d/2$, in 
\cite{SasakiNLKG06,SasakiNLKGbook} for (NLKG) with nonlocal nonliearities; \cite{LinSof15} for variable coefficients added in the cubic nonlinearity.

For a deeper discussion on well-posedness and large data scattering on $\R^d$, for energy subscritical and critical nonlinearities, 
(for both defocusing and focusing cases)
see \cite{IMN11,KSV11,Nakanishi99-low, Nakanishi98,Nakanishi99,Nakanishi01,NS11}. 
More recent small data results - such as low dimensions results - can be deduced from the references above.
We do not comment the large data problem since it will be handled in an ongoing work.

Existence of solutions for Nonlinear Klein-Gordon posed on tori and spheres have been studied, mainly in papers of Delort 
\cite{DS-NLKGtorispheres,Delort-NLKGtorus}, 
Delort-Szeftel \cite{DS-NLKGzoll,DS-NLKGspheres}, 
Fang-Zhang \cite{FZ10-NLKGtori}. 

There are also several results dealing with (NLKG) in various settings 
(with potentials and/or in other type of space structures). 
Several results about the decay of solutions and scattering for the cubic (NLW) and (NLKG) 
posed on Schwarzschild manifolds can be related to our framework. 
In fact, the equation is posed on a related static spacetime with product structure $$\R_t\times (2M,\infty) \times S^2.$$ 
But the spacelike hypersurface is equipped with a product metric of the form $g = dx^2 + c^2(x)g_0$, where 
$g_0$ is a metric on the compact manifold; in our case $c^2\equiv 1$.
However under some (strong) assumptions, one could reduce the study to (NLKG) similar to our problem but with additionnal 
terms that seem to be difficult to handle.
We give a non-exhaustive list of such references (for a deeper discussion of such frameworks involved in general relativity, 
we refer the reader to references given there): Blue-Soffer \cite{BlueSoffer, BlueSofferErr} and  handle decay estimates for (NLW), whereas existence and
asymptotical studies for (NLKG) are performed by Bachelot, Nicolas \cite{BchNic93,JPNLKGSch}, \cite{JPNLKGKerr} (on Kerr metrics).
We do not mention works about more general decay estimates, or works about the linear wave equation or wave maps posed on manifolds.

Since the Nonlinear Dirac equation (NLD) is intimately linked with (NLKG), 
results about it should be added to complete the references. 
Problems involving scattering for (NLD) are discussed for example in \cite{BH-H1,BH-H1/2, SasakiNLD3d, SasakiNLD1d} 
(see also \cite{BH-DKG} for systems of (NLD)-(NLKG)).
\\

The main tools to prove scattering are Strichartz estimates, whose proof will be detailed in Section \ref{strichartz}.
We will use the knowledge on the flat part to deduce Strichartz estimates on the whole product space, as it is 
performed in \cite{TVI}. 

\begin{remark}[Restrictions on $p$]
 Let us give details on the restrictions made on $p$. Those restrictions imply ones made on the dimensions $d,k$ at the beginning.
 \\
First, we need $p\geq 2$. In fact, smaller $p$ cannot be handled with our estimates. We would need more general Strichartz estimates 
that are not available with our argument, detailed in Section \ref{strichartz}. The homogeneous term will be estimated in 
$L^1_t L^2_{x,y}$ norms. By considering $p\geq 2$, we only deal with $p_c\geq 2$, giving $d+k\leq 6$. 
We also see that $1+4/d \leq 1+4/(d+k-2)$ which yields $k\leq2$.
\\

As for (NLS), it is quite natural to restrict $p\geq p_0$. In fact, considering data which are constant in their $y-$variables, 
it is easy to see that for $p<1+4/d$, the analysis is reduced to $L^2-$subcritical case on $\R^d$, for which no scattering in energy space seems 
to be available. 
\\
Besides as in \cite{TVI}, one can prove scattering for some $H^1-$supercritical cases $p\geq p_c$ in anisotropic Sobolev spaces of higher regularity 
(see Theorem \ref{main2}).
\end{remark}

\subsection{Notations}
We introduce some notations and definitions that will be useful in the paper.
\begin{notation}
  For any Lebesgue exponent $q\geq1$, we write $q'$ its dual: 
  $$\dfrac{1}{q}+\dfrac{1}{q'} = 1.$$
 \end{notation}
 We define admissibility for Schr\"odinger and Klein-Gordon:
 \begin{definition}
  \label{adm}
  A pair $(q,r)$ is {\bf admissible} if $2\leq r
  \leq\frac{2d}{d-2}$ ($2\leq r\leq
  \infty$ if $d=1$, $2\leq r<
  \infty$ if $d=2$)  
  and 
$$\frac{2}{q}= d\left( \frac{1}{2}-\frac{1}{r}\right).$$
 \end{definition}
 \vspace{2mm}
 \noindent
 We denote $\FT$ the Fourier transform with respect to the variable on $\R^d_x$ and $\FT^{-1}$ the inverse transform.
With the following partitions of the unity
$$ 1= \chi_0(\xi)+\sum_{j>0} \varphi_j(\xi), \forall \xi,$$
where
\begin{itemize}
 \myitem $\chi_0$ the cut-off function satisfying $\chi_0(\xi)=1$ for $|\xi|<1$ and $\chi_0(\xi)=0$ for $|\xi|>2$,
 \myitem $\varphi_j(\xi)=\chi_0(2^{-j}\xi)-\chi_0(2^{-j+1}\xi),$
\end{itemize}
we define the following operators
\begin{align*}
 P_0 f & := \mathcal{F}^{-1}\chi_0 \mathcal{F}(f), \\
 P_j f & := \mathcal{F}^{-1}\varphi_j \mathcal{F}(f), \quad \forall j >0,
\end{align*}
We denote by $\Sch'(\R^d)$ the set of all tempered distributions on $\R^d$.
\begin{notation}[Besov spaces on Euclidean spaces]
 Let $-\infty < s < \infty$, $0<q,r \leq \infty$. Then, for $0<q\leq \infty $, the Besov space $B^s_{q,r}$ is defined by
  $$B^s_{q,r}(\R^d) = \left\lbrace f\in \Sch'(\R^d) \Big|  
  \left(\sum_{j=0}^{\infty} 2^{jsr} \left\|P_j f\right\|_{L^q(\R^d)}^r\right)^{1/r}<\infty  
  \right\rbrace$$
\end{notation}
We denote by $C^\infty_0(\mathbb{M})$ 
the set of test functions on the manifold $\mathbb{M}$, which can be $\R^d$, $\mathcal{M}^k$ or  $\R^d\times\mathcal{M}^k$.
\begin{notation}[Sobolev spaces on Euclidean spaces, compact manifolds and product spaces]
Let $d\geq1$. For any $s\geq0$ and $1\leq q \leq \infty$, we write
 \begin{itemize}
  \myitem (Inhomogeneous Sobolev spaces) $W^{s,q}(\R^d)$ the completion of $C^\infty_0(\R^d)$  with respect to 
  $$\left\|\FT^{-1}((1+|\xi|^2)^{s/2}\FT(f))\right\|_{L^q(\R^d)}.$$
  
  \myitem (Homogeneous Sobolev spaces) $\dot{W}^{s,q}(\R^d)$ the completion of $C^\infty_0(\R^d)$  with respect to 
  $$\left\|\FT^{-1}(|\xi|^s\FT(f))\right\|_{L^q(\R^d)}.$$
 \end{itemize}
 When $q=2$, we will write $W^{s,2}(\R^d)=H^s(\R^d)$.
\\
For Sobolev spaces with integer derivatives $s\in \N$ in the $x-$variables, 
 an equivalent norm involving all derivatives of order smaller than $s$ could be used.
 \\
 
 Let $k\geq1$. 
\textcolor{black}{We write $\left\lbrace \lambda_j  \right\rbrace$ the eigenvalues of $-\Delta_y$, sorted in ascending order and taking in account the multiplicity,
$\left\lbrace \Phi_j(y) \right\rbrace$, the eigenfunctions associated with $\lambda_j$, that are }
\begin{equation}
         \label{basis}
         -\Delta_y \Phi_j = \lambda_j \Phi_j, \quad \lambda_j\geq 0,\; \forall j, 
        \end{equation}
then one has an orthonormal basis of $L^2\left(\mathcal{M}^k_y\right)$ given by \eqref{basis}.
\\

We introduce for any $q\in \N\setminus\left\lbrace 0\right\rbrace$, 
$$l_j^q=\left\lbrace \left\lbrace f_j \right\rbrace_{j\geq0} \Big|\left(\sum_{j\geq 0} |f_j|^q\right)^{1/q} <\infty  \right\rbrace. $$
We write $$f(y)=\sum_{j\geq 0} f_j \Phi_j(y) $$ the decomposition of any function $f: \mathcal{M}^k_y \longrightarrow \C$.
 For any $s\geq 0$ and $1\leq q \leq \infty$, we define
 \begin{itemize}
  \myitem (Inhomogeneous Sobolev spaces) $W^{s,q}(\mathcal{M}^k)$ the completion of $C^\infty_0(\mathcal{M}^k)$  with respect to 
  $$\left\|(1+\left|\lambda_j\right|)^{s/2}f_j\right\|_{l^q_j}.$$
  
  \myitem (Homogeneous Sobolev spaces) $\dot{W}^{s,q}(\mathcal{M}^k)$ the completion of $C^\infty_0(\mathcal{M}^k)$  
  with respect to 
  $$\left\|\left|\lambda_j\right|^{s/2}f_j\right\|_{l^q_j}.$$
 \end{itemize}
 When $q=2$, we will write $W^{s,2}(\mathcal{M}^k)=H^s(\mathcal{M}^k)$.
\\

Finally, for our product spaces $\mathbb{M}= \R^d_x \times \mathcal{M}^k_y$ 
we will use the following notations in the specific case $q=2$.
We consider $\lambda_j$ defined in \eqref{basis}.
\\

 \begin{itemize}
  \myitem \textcolor{black}{${H}^{\theta}_{x}{H}^{\rho}_{y}= 
  \left(1-\Delta_{x}\right)^{-\theta/2}\left(1-\Delta_{y}\right)^{-\rho/2}L^2_{x,y},$ endowed with the natural norm.}
  
  \myitem \textcolor{black}{$\dot{H}^{\theta}_{x}\dot{H}^{\rho}_{y}=
  \left(-\Delta_{x}\right)^{-\theta/2}\left(-\Delta_{y}\right)^{-\rho/2}L^2_{x,y},$ 
   endowed with the natural norms.}
  
  \myitem $H^{\theta}_{x,y} = \left(1-\Delta_{x,y}\right)^{-\theta/2}L^2_{x,y},$ endowed with the natural norm.
 \end{itemize}
\end{notation}

\subsection{The results}
Let us write for any $(f,g)\in H^1_{x,y}\times L^2_{x,y}$
\begin{align}
 \label{free solution}
 S(t)\left(f,g\right) & = 
 \cos\left(t. \sqrt{1-\Delta_{x,y}}\right)f + \dfrac{\sin \left(t. \sqrt{1-\Delta_{x,y}}\right)}{\sqrt{1-\Delta_{x,y}}}g
 \\
 \label{free vector}
 \overrightarrow{S(t)}\left(f,g\right) &= \left(S(t)\left(f,g\right), \d_t S(t)\left(f,g\right) \right).
\end{align}
Then we prove
\begin{theorem}
 \label{main}
 \textcolor{black}{Let $1\leq d\leq 5$ and $k=1,2$ such that $3\leq d+k\leq 6$. 
 Consider $p_0\leq p\leq p_c$ given by \eqref{pc}-\eqref{p0}.
Then, there exists $\delta >0$ such that the Cauchy 
 problem \eqref{NLKG} has a unique global solution 
 $$u(t,x,y)\in C^0(\R, H^1_{x,y})\cap C^1(\R, L^2_{x,y}), \quad u(t,x,y) \in L^p(\R, L^{2p}_{x,y}),$$
 $$\left(\textrm{ and so we have }\quad \d_t u(t,x,y) \in C^0(\R,L^2_{x,y})\right) $$
 for any initial data $(f,g) \in H^1_{x,y}\times L^2_{x,y}$ such that $\|f\|_{H^1_{x,y}}+\|g\|_{L^2_{x,y}}< \delta.$}
 \\
 
  \textcolor{black}{Moreover there exist two couples $(f^+,g^+),(f^-,g^-)$ in $H^1_{x,y}\times L^2_{x,y}$ such that
 $$\lim_{t \rightarrow \pm \infty} 
 \left\| \overrightarrow{S(t)}\left(f^\pm,g^\pm\right) - \left(u(t,.), \d_t u(t,.)\right) \right\|_{H^1_{x,y}\times L^2_{x,y}} =0.$$}
\end{theorem}
We also give some $H^1-$supercritical cases for which an additional regularity is required 
in the $y-$variable.
\textcolor{black}{\begin{theorem}
 \label{main2}
Let $1\leq d \leq 5$, and $k\geq 1$, with $k\geq 2$ if $d=1$. We assume $p_0\leq p$ given by \eqref{p0} and that 
$$p \quad  \left\lbrace \begin{array}{l}
                           < \;\infty \; \textrm{ if }d=1, 2, \\ \\
                           \leq \; \dfrac{d^2+2d-4}{d^2-2d} \; \textrm{ if } d\geq 3.
                          \end{array}
  \right.$$
Then for all $\gamma > k/2$, there exists $\delta >0$ such that the Cauchy 
 problem \eqref{NLKG} has a unique global solution 
 $$u(t,x,y)\in C^0(\R, H^1_{x,y})\cap C^1(\R, L^2_{x,y}) $$ satisfying
 $$(1-\Delta_y)^{\gamma/2}u(t,x,y)\in L^p(\R, L^{2p}_{x}L^{2}_y),$$
 for any initial data $(f,g)$ such that 
 $\left(f,g\right) \in \left(H^{1,\gamma}_{x,y}\times H^{0,\gamma}_{x,y}\right)\cap\left( H^{1}_{x,y}\times L^{2}_{x,y}\right)$ 
and $$\|(1-\Delta_y)^{\gamma/2}f\|_{H^1_{x,y}}+\|(1-\Delta_y)^{\gamma/2}g\|_{L^2_{x,y}}< \delta.$$
 Moreover there exist two couples 
 $$(f^+,g^+),(f^-,g^-) \textrm{ in } \left(H^{1,\gamma}_{x,y}\times H^{0,\gamma}_{x,y}\right)\cap\left( H^{1}_{x,y}\times L^{2}_{x,y}\right)$$ such that
 $$\lim_{t \rightarrow \pm \infty} 
 \left\|  \overrightarrow{S(t)}\left(f^\pm,g^\pm\right) - \left(u(t,.), \d_t u(t,.)\right) \right\|_{{H^{1,\gamma}_{x,y}\times H^{0,\gamma}_{x,y}}} =0.$$
\end{theorem}}
The main difference between both theorems, is that 
\textcolor{black}{Theorem \ref{main2} gives scattering in higher order Sobolev spaces but for larger $k$ and for a wider range of $p$. In particular, 
when $k=1,2$ we can consider some $H^1-$supercritical cases
$$p_c< p, \quad \textrm{ with } \;p \leq \dfrac{d^2+2d-4}{d^2-2d}, \quad \;\textrm{ if }\; 3 \leq d\leq 5,$$
whereas for $k\geq 3$, since $p_c < p_0$ we consider $H^1-$supercritical cases on the whole product space that are $L^2-$subcritical 
on $\R^d$.}

\begin{remark}[Comparison with the small data theory for (NLS)]
For Theorems \ref{main}, one can see that no additional regularity is required. Thus, $H^1-$scattering can be proved 
only making use of Lebesgue spaces, for some energy subcritical and critical nonlinearities; the latter beeing the most interesting case.
As an example, consider $\R^3\times \mathcal{M}^1$ and $\R^2\times \mathcal{M}^2$ for which the cubic nonlinerity is $H^1_{x,y}-$critical.
In \cite{TVI} one requires the smallness of the data in $H^0_{x_1,x_2}H^{1+\eps}_{x_3,y}$ or $H^0_{x_1,x_2}H^{1+\eps}_{y_1,y_2}$ 
and scattering is proved in those anisotropic spaces. For the mass-energy-critical exponents on $\R^d\times \mathcal{M}^2$, 
a more recent result from \cite{TarulliRdM2} gives $H^1-$scattering assuming smallness in anisotropic spaces with $1+\eps$ derivatives in $y$.
Small $(f,g)$ in $H^1_{x,y}\times L^2_{x,y}$ is enough in our case to prove $H^1-$scattering. 
\\
Besides, our results do not depend on the geometry of the compact manifold (more general $\mathcal{M}^k$ are allowed whereas the 
 explicit structure of the torus $\mathbb{T}^2$ was necessary in Theorem 1.1 of \cite{HP14} for (NLS) posed on $\R\times \mathbb{T}^2$).
 \\
 We also notice that for the small data theory, we do not use the finite time of propagation property.
\end{remark}

\section{About the Strichartz estimates}
\subsection{The strategy: from Schr\"odinger (\cite{TVI}) to Klein-Gordon}
\label{strategy}
Let us focus on more general frameworks for the Nonlinear Schr\"odinger equation to make a comparison.
The study of well-posedness on some specific product spaces and 
for small data scattering were the first steps of such analysis. Some large data scattering results are also proved in specific cases
(see \cite{HP14, HPTV, TTNLS, TarulliRdM2, TVI, TVII} and references therein).
\\

Since our aim is to investigate the persistence of scattering properties for small data, we recall the 
strategy used in \cite{TVI} and make a quick comparison with the difficulties arising for the Nonlinear Klein-Gordon equation.
The proofs of small data scattering rely on global-in-time Strichartz estimates, deduced from the Euclidean case 
and extended, in some sense, to the product spaces.
\\

Consider
\begin{equation}
 \label{NLS}
 i \d_t u + \Delta_{x,y} u = F, \quad u_{|_{t=0}}(x,y) = f(x,y), \quad (t,x,y) \in \R_t \times \R^d_x \times \mathcal{M}^k_y,
\end{equation}
where $d\geq 2$, $k\geq 1$.
The authors of \cite{TVI} 
are able to prove the following global in time Strichartz estimates on the whole product spaces, 
with a simple $L^2-$norm in the $y-$variables
\begin{equation}
\label{NLSstrichartz}
 \|u\|_{L^{q_1}_t L^{r_1}_xL^2_y} \leq C \left[ \|f\|_{L^2_{x,y}} +  
 \|F\|_{L^{q_2'}_t L^{r_2'}_x  L^2_y} \right].
\end{equation}
But the anisotropic Lebesgue spaces in $x$ and $y$ variables bring problems, handled by considering Strichartz with derivatives.
\\

When one wishes to use the same strategy for \eqref{NLKG}, the main problem appears after decomposing the functions on the basis \eqref{basis}.
In fact, writing
\begin{multline*}
 u(t,x,y)=  \cos\left(t. \sqrt{1-\Delta_{x,y}}\right)f(x,y) 
+ \dfrac{\sin \left(t. \sqrt{1-\Delta_{x,y}}\right)}{\sqrt{1-\Delta_{x,y}}} g(x,y) \\
+ 
\int_0^t \dfrac{\sin \left((t-s). \sqrt{1-\Delta_{x,y}}\right)}{\sqrt{1-\Delta_{x,y}}} F(s,x,y) \; ds,
\end{multline*}
the solution of
\begin{equation*}
       \d^2_t u - \Delta_{x,y} u + u = F, \quad u_{|_{t=0}}(x,y)= f(x,y), \quad \d_t u_{|_{t=0}}(x,y)= g(x,y),
\end{equation*}
with $(t,x,y)\in \R_t \times \R^d_x \times \mathcal{M}^k_y$, one writes  
\begin{align*}
 u(t,x,y) & = \sum_j u_j(t,x) \Phi_j(y) \\
  F(t,x,y) & = \sum_j F_j(t,x) \Phi_j(y) \\
   f(x,y) & = \sum_j f_j(x) \Phi_j(y) \\
    g(x,y) & = \sum_j g_j(x) \Phi_j(y),
\end{align*} and
obtains the following decoupled system of ``flat'' equations
\begin{equation}
 \label{flatNLKG}
 \d^2_t u_j - \Delta_{x} u_j + u_j + \lambda_j u_j = F_j, \quad u_{j|_{t=0}}(x)= f_j(x), \quad 
 \d_t u_{j|_{t=0}}(x)= g_j(x).
\end{equation}
It is easy to see that the Strichartz estimates involving 
$$\left(\cos\left(t. \sqrt{1+\lambda_j-\Delta_x}\right), \quad 
\dfrac{\sin \left(t. \sqrt{1+\lambda_j-\Delta_x}\right)}{\sqrt{1+\lambda_j-\Delta_x}}\right)$$ 
will not be independent of $\lambda_j$. Therefore the whole estimate 
on $\R^d \times \mathcal{M}_y^k$ will be altered in the $y-$variables for general admissible pairs.

But, unlike the Schr\"odinger case, we will see that in the framework of Theorem \ref{main}, 
no control of $u$ in anisotropic Sobolev spaces will be needed: 
in fact, obtaining a Sobolev norm in the $y-$variables is a good point. 
Thanks to an appropriate choice of pairs and Sobolev embeddings, we will be able to deal with the Cauchy problem \eqref{NLKG} 
with Lebesgue spaces in $y$.

\begin{remark}
Let us notice that one could deal with a mass $m^2>0, \; m^2 \neq 1$ in \eqref{NLKG}
\begin{equation*}
       \d^2_t u - \Delta_{x,y} u + m^2 u = \pm |u|^{p-1} u, \quad u_{|_{t=0}}(x,y)= f(x,y), \quad \d_t u_{|_{t=0}}(x,y)= g(x,y).
\end{equation*}
The quantity $m^2$ should be fixed at the beginning so it would not bring any problem in all computations involved in the proofs.
\\

Two other remarks about this mass $m^2$ can be made: 
first, even with the Non\-linear Wave equation ($m^2=0$), one ends up with a system of Nonlinear Klein Gordon equations 
\eqref{flatNLKG} with $\lambda_j u_j$ instead of $(1+\lambda_j)u_j$.\\
Then dealing with $m^2=0$ is not easy on arbitrary $\mathcal{M}^k_y$ since the first eigenvalue $\lambda_0$ could be zero. Hence, 
the first equation of the system \eqref{flatNLKG} would be a Wave equation, for which the Strichartz estimates 
are different with different admissibility.
We do not deal with this case is this paper.
\end{remark}

\subsection{Strichartz estimates}
\label{strichartz}
We recall admissibility for the Klein-Gordon equation, which is the same as for Schr\"odinger in Definition \ref{adm}:
  A pair $(q,r)$ is {\bf admissible} if $2\leq r
  \leq\frac{2d}{d-2}$ ($2\leq r\leq
  \infty$ if $d=1$, $2\leq r<
  \infty$ if $d=2$)  
  and 
$$\frac{2}{q}= d\left( \frac{1}{2}-\frac{1}{r}\right).$$

 We also define an exponent that will be used in the Besov spaces, in Section \ref{strichartz}:
 \begin{notation}
  Consider $(q,r)$ an admissible pair given by Definition \ref{adm}. We then denote by $s$ the following exponent:
  \begin{equation}
  \label{s}
s = 1- \dfrac{1}{2} \left(\dfrac{d}{2}+1\right).\left(\dfrac{1}{r'}-\dfrac{1}{r}\right)= 
  1- \dfrac{1}{2} \left(\dfrac{d}{2}+1\right).\left(1-\dfrac{2}{r}\right). 
  \end{equation}
 \end{notation}

\begin{proposition}[Strichartz estimates]
\label{thm:Strichartz}
 Let $1\leq d \leq 5$, $k\geq 1$ and assume $k\geq 2$ if $d=1$. Consider $p_0\leq p$ given by \eqref{p0} and 
 $u$ given by
 \begin{multline*} u(t,x,y) = \cos\left(t. \sqrt{1-\Delta_{x,y}}\right)f(x,y) 
+ \dfrac{\sin \left(t. \sqrt{1-\Delta_{x,y}}\right)}{\sqrt{1-\Delta_{x,y}}} g(x,y)
\\+\int_0^t \dfrac{\sin \left((t-s). \sqrt{1-\Delta_{x,y}}\right)}{\sqrt{1-\Delta_{x,y}}} F(s,x,y) \; ds.
\end{multline*}
 for any $f,g,F$ in $\Sch(\R^d \times \mathcal{M}^k)$. Then
\\

\noindent
$(1)$ Assume $k\leq 2$ such that $3\leq d+k\leq 6$ and 
$p_0\leq p\leq p_c$ given by \eqref{pc}-\eqref{p0}. Then
 \begin{equation}
 \label{Point1}
\|u \|_{L^{p}_t L^{2p}_{x,y}} \leq C \left[ \|f\|_{H^1_{x,y}} + \|g\|_{L^2_{x,y}} + \|F\|_{L^{1}_t L^2_{x,y}} \right],
\end{equation}
 where $C>0$ depends on $p$, and might depend on $Vol(\mathcal{M}^k)$.
  \\
 
 
 \noindent 
 $(2)$ Assume $p_0\leq p$ given by \eqref{p0}. 
 Assume $p\leq \dfrac{d^2+2d-4}{d^2-2d}$ if $3\leq d \leq 5$. Then
 \begin{equation}
 \label{Point3}
\|u \|_{L^{p}_t L^{2p}_x L^2_y} \leq C \left[ \|f\|_{H^1_{x,y}} + \|g\|_{L^2_{x,y}} + \|F\|_{L^{1}_t L^2_{x,y}} \right],
\end{equation}
 where $C>0$ depends on $p$ and might depend on $Vol(\mathcal{M}^k)$.
\end{proposition}
 \begin{remark}[Energy estimates]
  \label{energy}
One also has
  $$\|u(t)\|_{H^1_{x,y}}+\|\d_t u(t)\|_{L^2_{x,y}} \leq C\left[ \|f\|_{H^1_{x,y}}+\|g\|_{L^2_{x,y}} +\int_0^t\|F(t)\|_{L^2_{x,y}}\right].$$
 \end{remark} 
 
The proof is divided into four steps:
\begin{enumerate}
 \item We first sketch the proof of Strichartz estimates for the propagator $S^1_x(t)$ given by
 $$S^1_x(t)\left(f,g\right) = 
 \cos\left(t. \sqrt{1-\Delta_{x}}\right)f + \dfrac{\sin \left(t. \sqrt{1-\Delta_{x}}\right)}{\sqrt{1-\Delta_{x}}}g.$$
 In that case, we make use of Besov spaces on $\R^d$.
 One can find various statements and proofs of Strichartz estimates for (NLKG)
 in \cite{Brenner84, Brenner85, GV_NLKG_I, GV_decay, GV_NLKG_II, IMN11, KT, NOpisa, NS11, Pecher85} 
 (see for example \cite{GV_decay, GV_NLKG_I, GV_NLKG_II, KT, NS11} for proofs).
 \\
 We will briefly sketch the proof for non-extremal pairs as it is done in \cite{NS11} with standard $TT^*$ method 
introduced by Ginibre and Velo for this equation. 
Extremal pairs are given by Keel and Tao with different methods, using Littlewood-Paley theory in \cite{KT}.
For more details, we refer the reader to the references given above.
\\
 \item We use embedding theorems between Sobolev spaces and Lebesgue spaces (classical references here are 
 \cite{Adams, Triebel_I, Triebel_II, Triebel_III}).
 \\
 \item We use a scaling argument to deduce $m-$dependent Strichartz estimates for 
 $S^{m^2}_x(t)$ where
 $$S^{m^2}_x(t)\left(f,g\right) = 
 \cos\left(t. \sqrt{m^2-\Delta_{x}}\right)f + \dfrac{\sin \left(t. \sqrt{m^2-\Delta_{x}}\right)}{\sqrt{m^2-\Delta_{x}}}g.$$
 \\
 \item We then can use the Fourier decomposition in the $y-$variables. Summing on the modes, we obtain Strichartz estimates 
 for $S(t)$ given by \eqref{free solution}. It is then possible to use embedding theorems for the compact manifold to work in Lebesgue spaces.
\end{enumerate}
\begin{remark}[The endpoint $p=2$, $q=2d/(d-2)$]
In our framework, we only need to handle the endpoint  for $d=4,5$. For lower $d$, it is easy to see that $p_0>2$. 
Thus, the results of \cite{KT} are enough to deal with the endpoint (see also \cite{MNO02,MNO03}). 
\end{remark}

\subsection{The results in the Euclidean case}
Consider 
\begin{equation}
\label{duhamel} 
u(t,x) = \cos\left(t. \sqrt{1-\Delta_x}\right)f(x) 
+ \dfrac{\sin \left(t. \sqrt{1-\Delta_x}\right)}{\sqrt{1-\Delta_x}} g(x)
+\int_0^t \dfrac{\sin \left((t-s). \sqrt{1-\Delta_x}\right)}{\sqrt{1-\Delta_x}} F(s) \; ds.
\end{equation}
Then we notice that $u$ satisfies
\begin{equation*}
 \d^2_t u - \Delta_{x} u +  u = F, \quad u_{|_{t=0}}(x)= f(x), \quad 
 \d_t u_{|_{t=0}}(x)= g(x). 
\end{equation*}
\\
The statement of the following proposition does not take in account non-sharp pairs and is not as general as in \cite{NS11} and other references. 
In fact we just state the estimates in simple spaces that we are going to use.

\begin{proposition}[Strichartz estimates for $S^{1}_x(t)$ in Besov Spaces (from \cite{NS11})]
Let $d\geq 1$ and consider an admissible pair $(q,r)$, given by Definition \ref{adm}, and $s$ as in \eqref{s}. 
Consider $u$ given by~\eqref{duhamel} for any $f,g,F$ in $\Sch(\R^d)$.
Then
\begin{equation}
\label{strichartzRd}
\left\|u\right\|_{L^{q}(\R) B^{s}_{r,2}(\R^d)} \leq C \left[ \|f\|_{H^1(\R^d)} + \|g\|_{L^2(\R^d)} 
+ \|F\|_{L^{1}(\R) L^2(\R^d)} \right]
\end{equation}
where $C>0$ depends only on the choice of the pair.
\end{proposition}
In the statement we only put $L^1_tL^2_{x}$ since we will not use more general spaces as it will be explained in Remark \ref{rmk:L1L2}.
\begin{proof}[Sketch of the proof, from \cite{NS11}]
\textcolor{black}{We just sketch the proof for non-extremal pairs, assuming a dispersion inequality. 
We proceed as in the self-contained proof from \cite{NS11}, with standard $TT^*$ method.}
We will not recall intermediate results with stationary-non stationary phase methods.
\\

We use the cut-off functions introduced in Section \ref{intro}.
\\
We denote $\chi_0$ a cut-off function equal to one when $\xi$ is close to zero and $\chi$ a cut-off function equal to one on $1/2<|\xi|<2$, 
but supported on $\R^d\setminus\left\lbrace 0\right\rbrace $. 
Denote
\begin{align*}
U_0(t) &:= e^{it\sqrt{1-\Delta_x}}\chi_0(\nabla_x)\\
 U_\lambda(t)&:= e^{it\sqrt{1-\Delta_x}}\chi\left(\dfrac{\nabla_x}{\lambda}\right), \; \lambda \geq 1,
\end{align*}
that are bounded in $L^2$, uniformly in $\lambda$. It is easy to see that for any $f,g,F$ in the Schwartz class, and some 
fixed $q,r$ the following statements are equivalent
\textcolor{black}{\begin{align}
 \label{i} &\|U_\lambda U_\lambda^* F\|_{L^q_tL^r_x}\leq C(\lambda)^2\|F\|_{L^{q'}_tL_x^{r'}}, \; \forall F \in \Sch(\R\times \R^d),\\
  \label{ii} &\|U^*_\lambda g\|_{L_x^2}\leq C(\lambda)\|g\|_{L^{q'}_tL_x^{r'}}, \; \forall g \in \Sch(\R\times \R^d),\\
   \label{iii} & \|U_\lambda f\|_{L^q_tL^r_x}\leq  C(\lambda)\|f\|_{L^2_x}, \; \forall f \in \Sch(\R^d).
\end{align}}
\vspace{2mm}
\begin{center}
\begin{tabular}{p{12cm}}
  \it \textbf{Claim 1} : Let $d\geq1$, $2<q\leq \infty$, $2\leq r \leq \infty$ with $(q,r)$ admissible as in Definition \ref{adm}, 
 that is $2/q =d(1/2-1/r)$. Then, for any function $f$ in the Schwartz class, one has 
 \\[2mm]
 
 \hspace{2cm}$\|U_\lambda f\|_{L^q_tL^r_x}\leq C \left(1+\lambda^2\right)^\sigma \|f\|_{L^2_x}, \quad \lambda =0 \textrm{ or } \lambda \geq 1,$
 \\[2mm] \it where $\sigma =\dfrac{1}{2}\left(\dfrac{d}{2}+1\right)\left(\dfrac{1}{r'}-\dfrac{1}{r}\right)$.
\end{tabular}
\end{center}
\vspace{2mm}
\textbf{Proof of Claim 1:} One has 
$$\|U_\lambda f\|_{L^2_x}\leq C \|f\|_{L^2_x}.$$
\textcolor{black}{We just get for granted the following dispersion estimate
using stationary phase methods and Young's inequality (\cite{Brenner84,Pecher85,GV_decay,GV_NLKG_I,GV_NLKG_II,NS11})}
$$\|U_\lambda f\|_{L_x^\infty} \leq C_0 \lambda^{1+d/2}t^{-d/2} \|f\|_{L_x^1}.$$
Interpolating, one gets
$$\|U_\lambda f\|_{L_x^r} \leq {C_1} \left(\lambda^{1+d/2}t^{-d/2}\right)^{\frac{1}{r'}-\frac{1}{r}} \|f\|_{L_x^{r'}}.$$
As a consequence
$$\|U_\lambda U_\lambda^* F (t)\|_{L^r_x} \leq C \int_{-\infty}^{+\infty} 
\left(\lambda^{1+d/2}(t-s)^{-d/2}\right)^{\frac{1}{r'}-\frac{1}{r}} \|F (s)\|_{L^{r'}_x} ds$$
and Hardy-Littlewood-Sobolev inequality writes
$$\left\| f * \dfrac{1}{|t|^\alpha}\right\|_{L^s_t} \leq C(\rho,s,\alpha)\|f\|_{L^\rho_t}, 
\quad 1+1/s=\alpha + 1/\rho, \;\alpha \in (0,1).$$
Using this inequality, one gets
\begin{align*}
 \left(\int_\R \left|\|U_\lambda U_\lambda^* F (t)\|_{L^r_x}\right|^q\right)^{1/q}
 &\leq C' \lambda^{\left(\frac{d}{2}+1\right)\left(\frac{1}{r'}-\frac{1}{r}\right)} \left(\int_\R  \left|\int_{-\infty}^{+\infty}
 \dfrac{1}{(t-s)^{d/2\left(\frac{1}{r'}-\frac{1}{r}\right)}} \|F(s)\|_{L^{r'}_x} ds \right|^q dt \right)^{1/q}, \\
  \|U_\lambda U_\lambda^* F (t)\|_{L_t^qL^r_x}  &\leq C' \lambda^{\left(\frac{d}{2}+1\right)\left(\frac{1}{r'}-\frac{1}{r}\right)}
 \|F\|_{L^\rho_tL^{r'}_x},
\end{align*}
with $\alpha = d/2\left(\frac{1}{r'}-\frac{1}{r}\right)$, and $\rho=q'$.
\textcolor{black}{Noticing
that for $\lambda \geq 1$, one can bound $\lambda$ by $\sqrt{1+\lambda^2}$, we get~\eqref{i} for admissible pairs, except the endpoint $q=2$,
with $$C(\lambda)^2=C'\left(1+\lambda^2\right)^{\frac{1}{2}\left(1+\frac{d}{2}\right)\left(\frac{1}{r'}-\frac{1}{r}\right)}.$$ 
Now that \eqref{i} is proved, one also has \eqref{ii} and \eqref{iii} for admissible pairs with $C(\lambda)$.}
\begin{flushright}
 $\blacksquare$
\end{flushright}

Let us recall that Besov and Sobolev norms can be expressed with the partition of the unity given in the introduction, as
\begin{align*}
 \|f\|_{B^\sigma_{r,2}} & :=  \|P_0 f\|_{L^r} + \left(\sum_{j>0}2^{2\sigma j} \|P_jf\|_{L^r}^2\right)^{1/2} \\
 & \simeq \left\| 2^{\sigma j} \|P_jf\|_{L^r}\right\|_{l^2_j}, \\
 \|f\|_{H^\sigma} & := \|P_0 f\|_{L^2} + \left(\sum_{j>0}2^{2\sigma j} \|P_jf\|_{L^2}^2\right)^{1/2}.
\end{align*}
Then, taking $\lambda \in \left\lbrace 0 \right\rbrace \cup \left\lbrace 2^j, \; j\in \N \right\rbrace $ and summing in square 
the estimates from \textbf{Claim 1}, one obtains
$$\|e^{it\sqrt{1-\Delta}}f\|_{L^q_t B^0_{r,2}}\leq C \|f\|_{H^\sigma}, $$
and so
$$\|e^{it\sqrt{1-\Delta}}f\|_{L^q_t B^{1-\sigma}_{r,2}}\leq C \|f\|_{H^1}. $$
Consider $u$ given by \eqref{duhamel}.
\begin{center}
\begin{tabular}{p{12cm}}
  \it \textbf{Claim 2}: Any solution of \\[2mm]
 \hspace{2cm} $\d_t^2 u -\Delta_x u +u = F, \quad u(0,x) =f(x), \; \d_t u(0,x)=g(x), $
  \\[2mm]
  in $\R\times \R^d$ satisfies the estimates 
  \\[2mm]
  \hspace{2cm} $\|u\|_{L^q_t B^s_{r,2}} \leq C \left[\|f\|_{H^1}+\|g\|_{L^2}+\|F\|_{L^1_tL^2_x} \right],$ 
  \\[2mm] 
  where $(q,r)$ are admissible and $s$ is as in \eqref{s}.
\end{tabular}
\end{center}
\textbf{Proof of Claim 2:} The hompgeneous case $F=0$ is given before. Consider the case $F\neq 0$, but $f=g=0$. The whole 
estimate is obtained combining both cases. \\
For any $j \geq 0$, one wants to prove
$$\|P_j u\|_{L^q_tB^s_{r,2}}\leq C \|F\|_{L_t^1L^2_x}.$$
Noticing that for $C(\lambda,r)\simeq {\lambda^{\frac{1}{2}\left(\frac{d}{2}+1\right)\left(\frac{1}{r'}-\frac{1}{r}\right)}}$
$$\begin{array}{lll}
 U_0 : L^2 \rightarrow L^qL^r, & U_\lambda : C(\lambda,r)L^2 \rightarrow L^qL^r, & U_0U_0^* : L^1L^2 \rightarrow L^qL^r,\\
 \\
 U_0^* : L^1L^2 \rightarrow L^2, & U_\lambda^* : L^1L^2 \rightarrow L^2,&
 U_\lambda U_\lambda^* : C(\lambda,r) L^1L^2 \rightarrow L^qL^r,
\end{array}$$
and writing 
$$U_0 U_0^* F(t) = \int_\R e^{i(t-s)\sqrt{1-\Delta}}\chi_0(\xi)^2 F(s) ds, $$ 
we use Christ-Kiselev Lemma which allows to handle the integral on $[0,t]$, and the same argument can be used for the 
$U_\lambda U_\lambda^*$ term, providing the good weight in $\lambda$. The estimates are obtained 
proceeding as before.
\begin{flushright}
 $\blacksquare$
\end{flushright}
Finally, noticing that the endpoint $q=2$ is not given with the previous method, but is given in \cite{KT} 
and interpolating the associated estimates with the energy estimates,
one obtains the more general Strichartz estimates stated in the Proposition, by interpolation.
\end{proof}

\subsection{Statements of embedding theorems}
In this section, we state the embedding theorems without proofs and apply them to the estimates given by \eqref{strichartzRd}.
From \cite{Triebel_II, Triebel_III} one states
\begin{theorem}[Embedding theorems]
\label{embedding}
Let $d\geq 1$, $s>0$, $1\leq r,\rho \leq \infty$ and consider Besov $B^s_{r,2}(\R^d)$ and Lebesgue $L^{\rho}(\R^d)$ spaces
\begin{enumerate}
 \item $B^s_{r,2}(\R^d) \subset L^{r}(\R^d)$,
 \item $B^s_{r,2}(\R^d) \subset L^{r^*}(\R^d)$ if $r^*\geq 2$, where $s-d/r=-d/r^*$.
\end{enumerate}
\vspace{1mm}
Interpolation gives the Lebesgue spaces with exponents $\rho$ lying between $r$ and $r^*$, for $r\geq 2$\\

\hspace*{0.5mm} $(3)$ $B^s_{r,2}(\R^d) \subset L^{\rho}(\R^d)$ if $2\leq r\leq \rho \leq r^*$, where $s-d/r=-d/r^*$.
\end{theorem}
Point 1 of Theorem \ref{embedding} can be found in \cite{Triebel_II}(Section 2.3.2, Estimate (24) in Remark 3, p.97). \\
Point 2 can be found in \cite{Triebel_III} (Section 1.9.1, Theorem 1.73, Estimate (1.200), p.40) and can be deduced from a more general embedding 
involving Triebel-Lizorkin spaces somehow linked with Sobolev spaces. 
Such results can also be deduced from \cite{Adams, Triebel_I} and references therein.
\textcolor{black}{One can easily see that the relation between $r,r^*$ and $\rho$ are the same as for usual Sobolev embeddings which is quite natural 
since Besov spaces are obtained with interpolation between Sobolev spaces.}
Note that this theorem gives the Lebesgue spaces used in \cite{NS11} for the cubic NLKG on $\R$ (see Exercise 2.45, p. 75) 
and $\R^3$ (see Exercise 2.42, p. 74 and Lemma 2.46 p.77).
\\

We deduce
\begin{proposition}
Let $d\geq 1$. Consider an admissible pair $(q,r)$ given by Definition \ref{adm}, and $s$ as in \eqref{s}.
 Consider $u$ given by \eqref{duhamel}.
 Then for any $\rho$ such that $2\leq r\leq \rho \leq r^*$, where $s-d/r=-d/r^*$, we have
 \begin{equation}
 \label{S0}
\|u\|_{L^q_t L^\rho_x} \leq C \left[ \|f\|_{H^1_x} + \|g\|_{L^2_x} + \|F\|_{L^1_t L^2_x}\right],
 \end{equation}
\end{proposition}
The proof is immediate, applying Theorem \ref{embedding} to \eqref{strichartzRd}.
\textcolor{black}{\begin{remark}
 \label{rmk:L1L2}
Strichartz estimates proved in \cite{NS11} are more general than Proposition \ref{thm:Strichartz}, 
since they allow the source term to be in some
$L^{q'}B^{(1-s')}_{r',2}(\R^d)$, where $(q,r)$ admissible, $s'$ as in \eqref{s}. However, it is well known that embeddings of type 
$$W^{\sigma,r}(\R^d)\subset B^{\sigma}_{r,2}(\R^d), \; 1<r\leq 2, $$ are valid and the only way to obtain a Lebesgue space is to consider $\sigma=0$. 
Thus, the only $(q',r')$ giving Lebesgue exponents for the source term is $(1,2)$ with $s'=1$.
\end{remark}}

\subsection{The scaling argument for $m^2\neq 1$}
All the estimates enunciated before are proved for $m^2=1$.
As explained in Section \ref{strategy}, they need to be adapted to different masses, which is performed using a scaling argument.
\\

Consider $\lambda >0$ and 
\begin{multline}
\label{duhamellambda} u_\lambda(t,x) = \cos\left(t. \sqrt{\lambda-\Delta_x}\right)f_\lambda(x) 
+ \dfrac{\sin \left(t. \sqrt{\lambda-\Delta_x}\right)}{\sqrt{\lambda-\Delta_x}} g_\lambda(x) \\
+\int_0^t \dfrac{\sin \left((t-s). \sqrt{\lambda-\Delta_x}\right)}{\sqrt{\lambda-\Delta_x}} F_\lambda(s) \; ds.
\end{multline}
One sees that $u_\lambda(t,x) = u\left(\sqrt{\lambda}t, \sqrt{\lambda}x\right)$ where u is given by \eqref{duhamel} and that 
 it satisfies
 \begin{equation}
 \label{NLKGlambda}
  \d^2_t u_\lambda - \Delta_{x} u_\lambda +  \lambda u_\lambda = F_\lambda, \quad u_{\lambda|_{t=0}}(x)= f_\lambda(x), \quad 
 \d_t u_{\lambda|_{t=0}}(x)= g_\lambda(x), 
 \end{equation}
where 
        $$ f_\lambda(x) =  f\left(\sqrt{\lambda}x\right),
       g_\lambda(x) = \sqrt{\lambda}g\left(\sqrt{\lambda}x\right),
       F_\lambda(t,x) = \lambda F\left(\sqrt{\lambda}t, \sqrt{\lambda}x\right).$$
 \\      
Then
\begin{proposition}
 Let $d\geq 1$ and consider some $(q,r)$ admissible as in Definition \ref{adm}, $s$ given by \eqref{s} and $\rho$ such that $2\leq r\leq \rho \leq r^*$, where $s-d/r=-d/r^*$.
 Consider $u$ given by \eqref{duhamel} for which \eqref{S0} holds.
 Then for $u_\lambda$ given by \eqref{duhamellambda}, one has 
 \begin{equation}
  \label{Slambda}
\lambda^{\frac{1}{2}\left(\frac{d}{\rho}+\frac{1}{q}+1-\frac{d}{2}\right)}\|u_\lambda\|_{L^{q}_t L^{\rho}_x} 
\leq C \left[ \sqrt{\lambda}\|f_\lambda\|_{L^2_x}+  \|f_\lambda\|_{\dot{H}^1_x} + \|g_\lambda\|_{L^2_x} + 
 \|F_\lambda\|_{L^1_t L^2_x}\right].  
 \end{equation}
\end{proposition}
\begin{proof}
 Considering $u$ and $u_\lambda$, respectively given by \eqref{duhamel} and \eqref{duhamellambda}, an easy computation shows 
 that
 \begin{itemize}
  \myitem $\|u_\lambda\|_{L^{q}L^{\rho}} = \lambda^{\frac{-1}{2}\left(\frac{d}{\rho}+\frac{1}{q}\right)}
  \|u\|_{L^{q}L^{\rho}},$
  \\
  \myitem $\|f_\lambda\|_{H^1} = \|f_\lambda\|_{L^2}+\|f_\lambda\|_{\dot{H}^1} = \lambda^{\frac{-d}{4}}\|f\|_{L^2} 
  + \lambda^{\frac{1}{2}-\frac{d}{4}} \|f\|_{\dot{H}^1}$, \\
\myitem $\|g_\lambda\|_{L^2} =  \lambda^{\frac{1}{2}-\frac{d}{4}} \|g\|_{L^2}$,  \\
  \myitem $\|F_\lambda\|_{L^1 L^2} = 
  \lambda^{\frac{1}{2}-\frac{d}{4}}\|F\|_{L^1 L^2}$,
 \end{itemize}
and combining those estimates, one gets \eqref{Slambda}.
\end{proof}

\subsection{End of the proof of Proposition \ref{thm:Strichartz}}
\subsubsection{General computations.}
The next step is to consider \eqref{flatNLKG} taking $\lambda = 1+\lambda_j$ with $\lambda_j$ given in \eqref{basis} and 
 follow the strategy described in Section \ref{strategy}.
\\

We recall the system of equations \eqref{flatNLKG} after a decomposition on \eqref{basis}
 \begin{align*}
 u(t,x,y) & = \sum_j u_j(t,x) \Phi_j(y) \\
  F(t,x,y) & = \sum_j F_j(t,x) \Phi_j(y) \\
   f(x,y) & = \sum_j f_j(x) \Phi_j(y) \\
    g(x,y) & = \sum_j g_j(x) \Phi_j(y),
\end{align*}
that is
\begin{equation*}
 \d^2_t u_j - \Delta_{x} u_j + u_j + \lambda_j u_j = F_j, \quad u_{j|_{t=0}}(x)= f_j(x), \quad 
 \d_t u_{j|_{t=0}}(x)= g_j(x). 
\end{equation*}

From \eqref{Slambda}, with $\lambda = 1+\lambda_j$, one has for an appropriate choice of pairs that will be given later
\begin{equation*}
    \left(\lambda_j+1\right)^{\frac{1}{2}\left(\frac{d}{\rho}+\frac{1}{q}+1-\frac{d}{2}\right)} \|u_j\|_{L^q_tL^\rho_x}  \leq  C \left[ 
    \left(\lambda_j+1\right)^{1/2} \|f_j\|_{L^2_x} +
  \|f_j\|_{\dot{H}^1_x}+\|g_j\|_{L^2_x} + \|F_j\|_{L^1_tL^2_x}\right].
 \end{equation*}
Then, summing in $j$ the square as in \cite{TVI} one obtains
\begin{equation*}
  \left\|\left(\lambda_j+1\right)^{\frac{1}{2}\left(\frac{d}{\rho}+\frac{1}{q}+1-\frac{d}{2}\right)}u_j\right\|_{l^2_jL^q_tL^\rho_x}  
  \leq C \left[ \left\|\left(\lambda_j+1\right)^{1/2} f_j\right\|_{l^2_jL^2_x} +
  \|f_j\|_{l^2_j\dot{H}^1_x}+\|g_j\|_{l^2_jL^2_x} + \|F_j\|_{l^2_jL^1_tL^2_x}\right]. 
\end{equation*}
 Minkowski inequality for the left handside and for the source term (since $\max(1,2)\leq 2 \leq \min (q,\rho)$) gives
 \begin{equation*} 
 \left\|\left(\lambda_j+1\right)^{\frac{1}{2}\left(\frac{d}{\rho}+\frac{1}{q}+1-\frac{d}{2}\right)}u_j\right\|_{L^q_tL^\rho_xl^2_j}  
 \leq C \left[ \left\|\left(\lambda_j+1\right)^{1/2} f_j\right\|_{L^2_xl^2_j} +
  \|f_j\|_{\dot{H}^1_xl^2_j}+\|g_j\|_{L^2_xl^2_j} + \|F_j\|_{L^1_tL^2_xl^2_j}\right]. 
  \end{equation*}
  
Using Plancherel identity we notice that
\begin{align*}
&  \underbrace{\|\sqrt{1+\lambda_j}f_j\|_{l^2_jL^2_x}}_{\simeq \|f\|_{L^2_{x,y}} + \||\lambda_j|^{1/2}{f}\|_{L^2_{x,y}}} 
 +  \underbrace{\|f_j\|_{l^2_j\dot{H}^1_x}}_{\simeq \|\d_x f\|_{L^2_{x,y}}} \simeq \|f\|_{H^1_{x,y}}.
 \end{align*}
Then thanks to the decomposition on \eqref{basis}, one is able to handle the $y-$variable to obtain
  \begin{align*}
   \left\|(1-\Delta_y)^{\frac{1}{2}\left(\frac{d}{\rho}+\frac{1}{q}+1-\frac{d}{2}\right)}u\right\|_{L^q_tL^\rho_xL^2_y} & \leq C \left[ 
  \|f\|_{{H}^1_{x,y}} + \|g\|_{L^2_{x,y}}+\|F\|_{L^1_tL^2_xL^2_y}\right] \\  
   \left\|u\right\|_{L^q_tL^\rho_x H^{\gamma}_y} & 
   \leq C \left[\|f\|_{{H}^1_{x,y}} + \|g\|_{L^2_{x,y}}+\|F\|_{L^1_tL^2_xL^2_y}\right],
  \end{align*}
 where \begin{equation*}
        \gamma = \left(\frac{d}{\rho}+\frac{1}{q}+1-\frac{d}{2}\right)\geq 0, \textrm{ since } 2\leq r \leq \rho \leq r^*. 
       \end{equation*}
\subsubsection{Restrictions to prove \eqref{Point1} and \eqref{Point3}: ideas and remarks.}
Let us now focus on the exponents in Proposition \ref{thm:Strichartz}. 
We make the following requirements on the exponents:
\textcolor{black}{\begin{description}
 \item[(i)] We fix $q=p$, then $r$ is uniquely defined. We then find all possible $p$ for which $\rho=2p$.
 \item[(ii)] $p$ should also satisfy $H^\gamma(\mathcal{M}^k) \subset L^{2p}(\mathcal{M}^k)$.
\end{description}}
We first notice that for $(q,\rho)=(p,2p)$, $H^\gamma(\mathcal{M}^k) \subset L^{2}(\mathcal{M}^k)$ where $\gamma$ is given by~\eqref{gamma}. 
Let us now make the following remark on \textbf{$(ii)$}: 
we introduce 
\begin{equation*}
 p_{sob}=\dfrac{d+2}{d+k-2}
\end{equation*}
 and notice that for every $k\geq 1$, $p_{sob}< p_c$. The Sobolev embedding 
$$H^{\gamma}(\mathcal{M}^k)\subset L^{2p}(\mathcal{M}^k) \textrm{ if } k \geq 2\gamma>0 \; \textrm{ and } \; 
\dfrac{k}{k-2\gamma}\geq 2p,$$
gives $p\geq p_{sob}$.
But if $k<2\gamma$, Morrey estimates are available and 
$$H^{\gamma}(\mathcal{M}^k)\subset L^\infty(\mathcal{M}^k).$$
$Vol(\mathcal{M}^k)$ is finite, 
\begin{equation}
 \label{infty}
 \|f\|_{L^q_y}\leq \|f\|_{L^\infty_y} \times Vol(\mathcal{M}^k)^{1/q}, \quad \forall q\geq 1,
\end{equation}
which is true for the particular case $q=2p$. Let us notice that since $p\geq2$ and
 $H^\gamma(\mathcal{M}^k) \subset L^{2}(\mathcal{M}^k)$ then $H^\gamma(\mathcal{M}^k) \subset L^{2p}(\mathcal{M}^k)$ 
 even for a manifold with infinite volume.
Hence 
$$H^\gamma (\mathcal{M}^k) \subset L^{2p}(\mathcal{M}^k) \quad \textrm{if} \quad k<2\gamma \; \textrm{ or } \; 
k \geq 2\gamma, \; \textrm{ and }\; \dfrac{2k}{k-2\gamma}\geq 2p.$$

As an example, consider $\R^3\times \mathcal{M}^1$. Point $(1)$ of Theorem \ref{main} gives energy scattering 
for \\$5/2\leq p \leq 3$ but the interval $[7/3\; , \; 5/2)$ is not covered. 
However we can deal with global existence and energy scattering when $7/3\leq p <5/2$ 
using \eqref{infty}. 
The only cases that need to handle two different regions for $p$ are $\R^d\times \mathcal{M}^1$ for $d=2,3$. In fact, 
in our framework, the only case allowed with $d=1$ is $\R\times \mathcal{M}^2$ handled with Point $(i)$ of Theorem \ref{main}.
And for $d=4,5$, it is easy to see that $ p_{sob}\geq p_0$: 
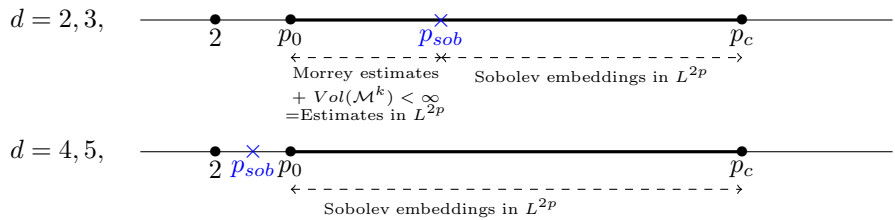
\begin{center}
\begin{tikzpicture}
\node[left] at (0,0) {$d=2,3, \quad $};
\draw (0,0) -- (10,0);
\draw[very thick,-] (2,0) -- (8,0);
\node at (1,0) {$\bullet$};
\node at (2,0) {$\bullet$};
\node at (8,0) {$\bullet$};
\node[text=blue] at (4,0) {$\times$};
\node[below] at (1,0) {$2$};
\node[below] at (2,0) {$p_0$};
\node[below] at (8,0) {$p_c$};
\node[below,text=blue] at (4,0) {$p_{sob}$};
\draw[dashed,<->](2,-0.5) -- (4,-0.5);
\draw[dashed,<->](4,-0.5) -- (8,-0.5);
\node[below] at (3,-0.5) {\tiny Morrey estimates};
\node[below] at (6,-0.5) {\tiny Sobolev embeddings in $L^{2p}$};
\node[below] at (3,-0.75) {\tiny + $Vol(\mathcal{M}^k)<\infty$ };
\node[below] at (3,-1){\tiny =Estimates in $L^{2p}$};
\end{tikzpicture}
\vspace{2mm}
\begin{tikzpicture}
\node[left] at (0,0) {$d=4,5, \quad $};
\draw (0,0) -- (10,0);
\draw[very thick,-] (2,0) -- (8,0);
\node at (1,0) {$\bullet$};
\node at (2,0) {$\bullet$};
\node at (8,0) {$\bullet$};
\node[text=blue] at (1.5,0) {$\times$};
\node[below] at (1,0) {$2$};
\node[below,text=blue] at (1.5,0) {$p_{sob}$};
\node[below] at (8,0) {$p_c$};
\node[below] at (2,0) {$p_0$};
\draw[dashed,<->](2,-0.5) -- (8,-0.5);
\node[below] at (4,-0.5) {\tiny Sobolev embeddings in $L^{2p}$};
\end{tikzpicture}
\end{center}
\subsubsection{Claims and proofs.}
Set $q=p$, then 
\begin{align}
 \label{r(p)}
 r&=\dfrac{2dp}{dp-4}\\
 \label{s(p)}
 s&=\dfrac{dp-d-2}{dp}\\
 \label{gamma}
 \gamma&=\dfrac{d+2+2p-dp}{2p},
\end{align}
and we notice that if $d=1$, we should restrict to $p\geq 4$. 
We recall that the exponent $p$ in our framework satisfies $p_0 \leq p\leq p_c$ that is
\begin{itemize}
 \myitem $d=1$, $p \geq 5$,
 \myitem $d=2$, $3 \leq p \leq 1+ 4/k$,
 \myitem $d\geq 3$, $\max \left(2, 1+4/d \right) \leq p \leq 1+4/(d+k-2)$,
\end{itemize}
and notice that to apply Theorem \ref{embedding} we need $s>0$ and $r^* = dr/(d-sr)\geq r \geq 2$. 
The condition $s>0$ is fulfilled in the following cases
\begin{itemize}
 \myitem $d=1$, $p> 3$,
 \myitem $d=2$, for all $p>2$,
 \myitem $d\geq 3$, for all $p> 1+ 2/d$,
\end{itemize}
which is always fulfilled for $p\geq p_0$.
\\

\begin{center}
\begin{tabular}{p{12cm}}
 \textbf{Claims: }\it Consider $1\leq d\leq 5$ and $p_0 \leq p,$ given by \eqref{p0}.
 \\[2mm]
 \textit{a. Then for any $(p,r)$ admissible as in \eqref{r(p)} and $s$ given by \eqref{s(p)}, one has
 $$B^s_{r,2}(\R^d)\subset L^{2p}(\R^d), \quad \textrm{ if }\quad p\leq \dfrac{d^2+2d-4}{d^2-2d} \; \textrm{ when }\;  d\geq 3.$$}
\noindent
 \textit{b. For any couple $(p,2p)$ satisfying Claim a., and for any $k=1,2$, $k=2$ if $d=1$,
 $$H^\gamma(\mathcal{M}^k) \subset L^{2p}(\mathcal{M}^k) \quad \textrm{ if } \quad p\leq p_c. $$
where $\gamma$ is given by \eqref{gamma}}.
 \end{tabular}
 \end{center}
\vspace{2mm}
 \textbf{Proof: }
 
\begin{description}
 \item[$d=1$] From \cite{NS11}, one gets for any $4\leq p \leq \infty$, $2\leq r \leq \infty$, given by \eqref{r(p)} and 
 $s=(p-3)/p$,
 $$\|u_j\|_{L^q_tL^{\rho}_x} \lesssim \|f_j\|_{H^1_x}+\|g_j\|_{L^2_x}+\|F_j\|_{L^1_tL^2_x}, $$
 provided that $$ \rho \geq r \; \textrm{ and } \; 
 \rho \;\textrm{ is such that } \; s\geq \dfrac{1}{r}-\dfrac{1}{\rho}. $$ 
 Replacing $\rho$ with $2p$, $r\leq 2p$ yields $p\geq p_0$ and 
 for any $p\geq p_0$, we have $$\dfrac{p-3}{p}\geq \dfrac{1}{r}-\dfrac{1}{2p}=\dfrac{p-5}{2p},$$
 which gives Claim a.
 \\
 We compute $\gamma$ for the couple $(q,\rho)=(p,2p)$ which gives 
 $\gamma=(3+p)/2p$. 
 It is then easy to see that for $k=2$
 $$H^\gamma(\mathcal{M}^k) \subset L^{2p}(\mathcal{M}^k) \textrm{ if } \dfrac{2k}{k-2\gamma}\geq 2p \textrm{ that is } p= p_c=5.$$ 
 which gives Claim b and c for $d=1$.
 \\
 
 \item[$d=2$] We compute $$sr = \dfrac{2(2p-4)}{2p-4}=2.$$ Thus, we have 
  $B^s_{r,2}(\R^2) \subset L^{\rho}(\R^2),$ for all $r\leq \rho <r^*=\infty$. So it is enough to check $2p \geq r$ which is true for any 
  $p\geq p_0$ and Claim a is proved.
  \\
We then see $\gamma = 2/p$ and $$H^{2/p}(\mathcal{M}^k) \subset L^{2p}(\mathcal{M}^k) 
\textrm{ if } k<\dfrac{4}{p} \;\textrm{ or }\; k\geq \dfrac{4}{p}  \;\textrm{ and }\; \dfrac{2k}{k-4/p}\geq 2p, $$
that is $p\leq p_c$.
\\

\item[$d\geq 3$] The restrictions $2\leq r\leq 2p\leq r^* = dr/(d-sr)$ gives $$1+\dfrac{4}{d}\leq p\leq \dfrac{d^2+2d-4}{d^2-2d},$$
and since we only deal with $p\geq2$, it is easy to check that the condition $$\dfrac{d^2+2d-4}{d^2-2d}\geq 2$$ 
allows us to deal with $d\leq 5$, which gives Claim a for $d\geq3$. 
We then see that $$H^\gamma(\mathcal{M}^k) \subset L^{2}(\mathcal{M}^k)\quad \textrm{if} \quad p\geq p_0$$ and
$$H^\gamma(\mathcal{M}^k) \subset L^{2p}(\mathcal{M}^k) \quad \textrm{ if }\quad 
k< 2\gamma \;\textrm{ or }\; k \geq 2\gamma  \;\textrm{ and }\; \dfrac{2k}{k-2\gamma} \geq 2p,$$ 
gives the condition $$ p\leq \dfrac{d+k+2}{d+k-2}=p_c.$$ We finally notice that $p_c\leq \dfrac{d^2+2d-4}{d^2-2d}$ 
in the framework of Theorem \ref{main}.
\end{description}

To obtain \eqref{Point3} in Proposition \ref{thm:Strichartz}, we restrict to the cases 
covered by Claim a and b, whereas for \eqref{Point1}, restrictions carried by Claim a and c will give the results.
\\

We sum up the restrictions in the following table
\\[2mm]
\noindent
\begin{center}
\begin{tabular}{|l|c|c|c|}
\hline
Conditions & \hspace{1cm}$d=1$\hspace*{1cm} &\hspace{1cm}$d=2$\hspace*{1cm} & $d\geq3$ \\
 \hline
			 \cline{2-4}
 $(q,\rho)=(p,2p)$  &\multicolumn{2}{c|}{$p_0\leq p$} 
			& $p_0\leq p \leq \dfrac{d^2+2d-4}{d^2-2d}$ \\
\hline &&&\\
Possible $p$ to have \eqref{Point3}&  $k\geq 2,$ & $k\geq 1$, &
$k\geq 1$, $d\leq 5$, 
\\ &&&\\ &$ p\geq 5$, & 
$p\geq 3$
&$p_0\leq p \leq \dfrac{d^2+2d-4}{d^2-2d}$
\\ &&&\\
\hline 
 \hline
$H^\gamma(\mathcal{M}^k) \subset L^{2p}(\mathcal{M}^k)$& \multicolumn{3}{c|}{$p_0\leq p\leq p_c$} \\ 
\hline &&&\\
Possible $p$ to have \eqref{Point1}& $k = 2,$ & $k=1,2$, &
$k=1,2$, $d\leq 5$, $d+k\leq 6$ 
\\ & $p=5$ & $3\leq p \leq 1+\dfrac{4}{k}$ & $p_0\leq p \leq p_c$ \\ &&&\\
\hline
\end{tabular}
\end{center}
 which concludes the proof of Proposition \ref{thm:Strichartz}.

\section{Proof of Theorem \ref{main}}
\subsection{Global existence of the solution}
Let us first remark that (at least local) existence theory is available on smooth ($C^\infty$) Riemannian manifolds without boundaries, 
of dimension larger than 3 (in view of \cite{KapitanskiiII,KapitanskiiIII}), for defocusing and energy critical nonlinearities. 
The finite speed propagation is a key point in the analysis.
We will not use those results here, since we also deal with focusing energy (sub-)critical cases and will perform a classical fixed point argument 
in the small data framework. For simplicity, we write  $L^p_tL^{2p}_{x,y}=L^p(\R,L^{2p}_{x,y})$. We recall that we only consider $p\geq2$.
\\
\begin{multline*}
\Phi_0 u(t) = \cos\left(t. \sqrt{1-\Delta_{x,y}}\right)f(x) 
+ \dfrac{\sin \left(t. \sqrt{1-\Delta_{x,y}}\right)}{\sqrt{1-\Delta_{x,y}}} g(x)
 + \\
\int_0^t \dfrac{\sin \left((t-s). \sqrt{1-\Delta_{x,y}}\right)}{\sqrt{1-\Delta_{x,y}}} (|u|^{p-1}u)(s) \; ds.   
\end{multline*}
Then \eqref{NLKG} is equivalent to $\Phi_0 u(t)=u(t)$.
Applying \eqref{Point1} with $F=|u|^{p-1}u$ on $\Phi_0 u(t)$ we get
\begin{align*}
\|\Phi_0 u(t)\|_{L^p_tL^{2p}_{x,y}}& 
\leq C \left[\|f\|_{H^1_{x,y}}+\|g\|_{L^2_{x,y}}+ \||u|^{p-1}u\|_{L^1_tL^2_{x,y}} \right] \\
 & \leq C \left[\|f\|_{H^1_{x,y}}+\|g\|_{L^2_{x,y}}+ \|u\|_{L^p_tL^{2p}_{x,y}}^p \right].
\end{align*}
We define for $R>0$
$$X_R=\left\lbrace \varphi | \|\varphi\|_{L^p_tL^{2p}_{x,y}}\leq R \right\rbrace,$$ and $\delta>0$ such that 
$\|f\|_{H^1_{x,y}}+\|g\|_{L^2_{x,y}}<\delta$, and will specify $\delta$ later.
Taking $R=2C\delta$, any $u\in X_{2C\delta}$ satisfies
$$\|\Phi_0 u(t)\|_{L^p_tL^{2p}_{x,y}} \leq C \left[\|f\|_{H^1_{x,y}}+\|g\|_{L^2_{x,y}}+ \|u\|_{L^p_tL^{2p}_{x,y}}^p \right] \leq 
2C\delta \left(\dfrac{1}{2}+ (2\delta)^{p-1}C^p \right).$$
Taking $\delta \leq \dfrac{1}{(2C)^{p/(p-1)}}$ one obtains that $\Phi_0$ maps $X_{2C\delta}$ into itself.
\\

\noindent
Consider two solution $v,w\in X_{2C\delta}$ with the same initial data. Then with Proposition \ref{thm:Strichartz}, we have
$$\|\Phi_0 v - \Phi_0 w\|_{L^p_tL^{2p}_{x,y}} \leq C \left\||v|^{p-1}v-|w|^{p-1}w \right\|_{L^1_tL^2_{x,y}}.$$
We can proceed as in \cite{CazCourant} to handle non integer nonlinearities.
$$\left| |v|^{p-1}v-|w|^{p-1}w \right| \leq \widetilde{C}\left(|v|^{p-1}+|w|^{p-1}\right)|v-w|.$$
$$\|\Phi_0 v - \Phi_0 w\|_{L^p_tL^{2p}_{x,y}} \leq \widetilde{C}\|v-w\|_{L^p_tL^{2p}_{x,y}}
\left[\|v\|_{L^p_tL^{2p}_{x,y}}^{p-1}+\|w\|_{L^p_tL^{2p}_{x,y}}^{p-1} \right].$$
Any $v,w$ in $X_{2C\delta}$ satisfy
$$\|\Phi_0 v - \Phi_0 w\|_{L^p_tL^{2p}_{x,y}} \leq \left(2^p\delta^{p-1} \max(C,\widetilde{C})^{p}\right) \|v-w\|_{L^p_tL^{2p}_{x,y}}.$$
Hence choosing $\delta$ such that $2^p\delta^{p-1} \max(C,\widetilde{C})^{p}<1$, $\Phi_0$ is a contraction on $X_{2C\delta}$.
\\

Finally, one can deduce global existence of the (unique) solution in $L^p(\R,L^{2p}_{x,y})$ for any $(f,g)\in H^1_{x,y}\times L^2_{x,y}$
such that $$\|f\|_{H^1_{x,y}}+\|g\|_{L^2_{x,y}}< \delta_0 = \min \left(\dfrac{1}{(2C)^{p/(p-1)}};
\dfrac{1}{\left(2\max(C,\widetilde{C})\right)^{p/(p-1)}}\right).$$

Then, using the energy estimate \eqref{energy} and $u\in L^p(\R, L^{2p}_{x,y})$, one gets a solution $u$ such that
$$u(t,x,y)\in C^0(\R, H^1_{x,y})\cap C^1(\R, L^2_{x,y}), \quad \textrm{ and so } \quad \d_t u(t,x,y) \in C^0(\R,L^2_{x,y}). $$
\subsection{Scattering results} Scattering is proved in the fashion of \cite{NS11}. 
We write
$$\d_t\begin{pmatrix}
         u \\ \d_t u
        \end{pmatrix}
 = \begin{pmatrix}
    0 & 1 \\ -\Delta_{x,y}+1 & 0
   \end{pmatrix}\begin{pmatrix}
         u \\ \d_t u
        \end{pmatrix} + \begin{pmatrix}
         0 \\ \pm |u|^{p-1}u
        \end{pmatrix},
$$
we get with $e^{tH}= \begin{pmatrix}
    \cos\left(t. \sqrt{1-\Delta_{x,y}}\right) & \dfrac{\sin \left(t. \sqrt{1-\Delta_{x,y}}\right)}{\sqrt{1-\Delta_{x,y}}} 
    \\ \\
    -\sin \left(t. \sqrt{1-\Delta_{x,y}}\right).\left(\sqrt{1-\Delta_{x,y}}\right) &  \cos\left(t. \sqrt{1-\Delta_{x,y}}\right)
   \end{pmatrix}$, \\[2mm] which is unitary on the energy space $H^1_{x,y}\times L^2_{x,y}$
\begin{align*}\begin{pmatrix}
         u \\ \d_t u
        \end{pmatrix}
  &= e^{tH}\begin{pmatrix}
         f \\ g
        \end{pmatrix}+       
   \int_0^te^{(t-s)H}\begin{pmatrix}
         0 \\ \pm |u|^{p-1}u
        \end{pmatrix}ds \\        
e^{-tH}\begin{pmatrix}
         u \\ \d_t u
        \end{pmatrix}
  &= \begin{pmatrix}
         f \\ g
        \end{pmatrix}+       
   \int_0^te^{(-s)H}\begin{pmatrix}
         0 \\ \pm |u|^{p-1}u
        \end{pmatrix}ds.
\end{align*}
Writing $ V(t)= e^{-tH}\begin{pmatrix}
         u \\ \d_t u
        \end{pmatrix}$, and considering $0<\tau<t$ it is easy to check that
   $$\|V(t)-V(\tau)\|_{H^1\times L^2}\leq C \int_{\tau}^{t} \||u|^{p-1}u (s)\|_{L^2} ds \leq C \|u\|^p_{L^p([\tau,t],L^{2p})},$$
   where the latter norm tends to zero when $t,\tau$ tends to $\pm \infty$ since the solution belongs to 
   $L^p(\R,L^{2p}_{x,y})$. Therefore, there exist 
   $(f^\pm,g^\pm)\in H^1_{x,y}\times L^2_{x,y}$ such that $V(t)\rightarrow \begin{pmatrix}
         f^\pm \\ g^\pm
        \end{pmatrix}$ in $H^1\times L^2$ as $t\rightarrow \pm \infty$.

\section{Proof of Theorem \ref{main2}}
As it is noticed in \cite{TVI, TarulliRdM2}, $(1-\Delta_y)^{\gamma/2}$ commutes with the linear part of \eqref{NLKG} on $\R^d\times \mathcal{M}^k$ for any $\gamma \geq 0$. 
Let us consider the assumptions for which Proposition \ref{thm:Strichartz} holds. 
From point $(2)$ of Proposition \ref{thm:Strichartz}, we have
\begin{align}
\label{S2'}&
\left\|\cos\left(t. \sqrt{1-\Delta_{x,y}}\right)(1-\Delta_y)^{\gamma/2}f\right\|_{L^{p}_t L^{2p}_{x}L^2_y} 
\\ \notag + & \left\|\dfrac{\sin \left(t. \sqrt{1-\Delta_{x,y}}\right)}{\sqrt{1-\Delta_{x,y}}} 
(1-\Delta_y)^{\gamma/2}g \right\|_{L^{p}_t L^{2p}_{x}L^2_y}\\
\notag
+& \left\|\int_0^t \dfrac{\sin \left((t-s). 
\sqrt{1-\Delta_{x,y}}\right)}{\sqrt{1-\Delta_{x,y}}} (1-\Delta_y)^{\gamma/2}F(s)\; ds \right\|_{L^{p}_t L^{2p}_{x}L^2_y} \\ 
\notag \\
\notag
&\leq C \left[ \|(1-\Delta_y)^{\gamma/2}f\|_{H^1_{x,y}} + \|(1-\Delta_y)^{\gamma/2}g\|_{L^2_{x,y}} 
+ \|(1-\Delta_y)^{\gamma/2}F\|_{L^{1}_t L^2_{x,y}} \right],
\end{align}
 where $C>0$ depends on the choice of the pairs.
The key point is that for $\gamma > k/2$, 
$H^{\gamma}(\mathcal{M}^k_y)$ is an algebra. Therefore, 
considering a product source term 
$F(t,x,y) = \prod_{i=1}^{I} u_i(t,x,y),$
one can write for $I \in \N$
$$\left\| \prod_{i=1}^{I} u_i \right\|_{H^\gamma_y} \leq C \prod_{i=1}^I \| u_i\|_{H^\gamma_y}.$$
Hence for $u_i \in \left\lbrace u, \overline{u}\right\rbrace,$ $I=p$
$$\left\|\prod_{i=1}^{p} \|u_i\|_{H^\gamma_y} \right\|_{L^1_tL^2_{x,y}} \leq 
C \left\|\| u\|_{H^\gamma_y}^p\right\|_{L^1_tL^2_x} \leq C \|u\|_{L^p_tL^{2p}_xH^\gamma_y}^p.$$
\textcolor{black}{The rest of the proof of global existence follows easily by using same arguments as for the proof of Theorem \ref{main} with 
$(1-\Delta_y)^{\gamma/2}u$. Note that the energy estimate together with 
$$\|u\|^p_{L^p_tL^{2p}_{x,y}}\lesssim \|u\|^p_{L^p_tL^{2p}_{x}L^\infty_y}
\lesssim \|u\|^p_{L^p_tL^{2p}_{x}H^{\gamma}_y} $$ gives the continuity in time of the solution.
\\
With the estimate, we see that writing $ V(t)= e^{-tH}\begin{pmatrix}
         u \\ \d_t u
        \end{pmatrix}$, and considering $0<\tau<t$ it is easy to check that
   \begin{align*}
    \|V(t)-V(\tau)\|_{\left(H^{1,\gamma}_{x,y}\times H^{0,\gamma}_{x,y}\right)\cap\left(H^1_{x,y}\times L^2_{x,y}\right)}
    &\leq C \int_{\tau}^{t} \||u|^{p-1}u (s)\|_{L^2_{x,y}} ds \\ & \leq C \|u\|^p_{L^p([\tau,t],L^{2p}_{x,y})}
   \leq C  \|u\|^p_{L^p([\tau,t]L^{2p}_{x}H^{\gamma}_y)},
   \end{align*}
   where the latter norm tends to zero when $t,\tau$ tends to $\pm \infty$ since the solution belongs to 
   $L^p(\R,L^{2p}_{x}H^\gamma_y)$, $\gamma>k/2$. Therefore, there exist 
   $(f^\pm,g^\pm)\in \left(H^{1,\gamma}_{x,y}\times H^{0,\gamma}_{x,y}\right)\cap\left(H^1_{x,y}\times L^2_{x,y}\right)$ 
   such that $V(t)\rightarrow \begin{pmatrix}
         f^\pm \\ g^\pm
        \end{pmatrix}$ in $\left(H^{1,\gamma}_{x,y}\times H^{0,\gamma}_{x,y}\right)\cap\left(H^1_{x,y}\times L^2_{x,y}\right)$ as $t\rightarrow \pm \infty$.}
\\[5mm]
\textbf{Acknowledgments.} The authors wish to thank Thomas Duyckaerts for several interesting discussions and clarifications about the problem. 
\bibliographystyle{amsplain}
\bibliography{biblioNEW}
\end{document}